\documentclass[11pt]{article}

\usepackage[T1]{fontenc}
\usepackage{lmodern}
\usepackage{amsmath,amssymb,amsthm,mathtools}
\usepackage[margin=1in]{geometry}
\usepackage{microtype}
\usepackage{booktabs}
\usepackage{enumitem}
\usepackage{xcolor}
\usepackage[hidelinks]{hyperref}
\hypersetup{
  pdftitle={The Kernel Deficit Dominates Twice the Hull Deficit: A Sharp Strengthening of Nakano's Inequality},
  pdfauthor={Dakota Charles Baker}
}
\usepackage[nameinlink,noabbrev]{cleveref}

\newcommand{\Ssup}{S}
\newtheorem{theorem}{Theorem}[section]
\newtheorem{lemma}[theorem]{Lemma}

\newtheorem{proposition}[theorem]{Proposition}
\theoremstyle{definition}

\theoremstyle{remark}
\newtheorem{remark}[theorem]{Remark}

\newcommand{\R}{\mathbb{R}}
\newcommand{\conv}{\operatorname{conv}}

\newcommand{\Per}{\operatorname{Per}}
\newcommand{\Sor}{\operatorname{Sor}}
\newcommand{\area}[1]{\lvert #1\rvert}
\newcommand{\inner}[2]{\langle #1,#2\rangle}
\newcommand{\defeq}{\coloneqq}
\newcommand{\eps}{\varepsilon}
\newcommand{\intr}{\operatorname{int}}
\newcommand{\extr}{\operatorname{ext}}
\newcommand{\Ind}{\operatorname{Ind}}
\newcommand{\relint}{\operatorname{relint}}

\title{\bfseries The Kernel Deficit Dominates Twice the Hull Deficit\\
  \large A sharp strengthening of Nakano's inequality}
\author{Dakota Charles Baker\\
  \href{mailto:dcb@dakotacharlesbaker.com}{\texttt{dcb@dakotacharlesbaker.com}}}
\date{Preprint v6, August 24, 2026}

\begin{document}

\maketitle

\begin{abstract}
A point lies in the kernel of a polygon if it can see the entire polygon.
Thus the kernel measures how much of the polygon is available to a single
guard, while the convex hull measures how far the polygon is from being
convex.  We prove that these two losses are linked by a sharp factor of two:
the area lost between a polygon and its kernel dominates twice the
kernel-weighted area missing from the polygon's convex hull.  In terms of
Sibley's guard-point ratio $G$ and exterior ratio $E$, which we denote by
$A$, the result is
$G\leq A/(2-A)$, which improves Nakano's inequality $G\leq A$ whenever
$A<1$.

The proof passes through a convex-body cap union.  From a compact convex body
$K$ and finitely many points whose convex hull contains it, we join every
point to $K$ and take the union $U$ of the resulting caps.  Cyclically sorting
the directed boundary edges of a polygonal $U$ produces a convex companion
$H$.  A boundary-reversal argument gives
$\area{H}+\area{U}\geq 2\area{\conv U}$, while a support-function identity
and Minkowski's mixed-area inequality give
$\area{U}^{2}\geq \area{K}\area{H}$.  Inner polygonal approximation handles
every positive-area compact convex $K$, while a separate null-area branch
covers points, segments, and all other lower-dimensional cases.  Both main
theorems have machine-checked Lean~4 proofs whose final statements were audited
against the informal statements after kernel checking.
\end{abstract}

\begin{quote}
\small
\textbf{Scope and certification.}
The cap-union theorem and the polygon theorem proved below are
\emph{FORMALIZED+}: their proofs are checked by the Lean kernel, and their
final Lean statements were audited against the informal statements after that
kernel checking.  The manuscript additionally passed separate referee and
citation-verification passes; as recorded in the declaration of generative AI
assistance at the end of the paper, those passes were automated checks and not
human peer review.  The final section on Nakano's perimeter optimum
$\alpha^*$ is deliberately separate.  It distinguishes rigorous bounds,
finite exact computations, verified reductions, conjectures, and the
still-unproved global steps;
in particular, it makes no claim to determine $\alpha^*$.
\end{quote}

\section{Introduction}
\label{sec:introduction}

Let $F\subset\R^2$ be the compact region enclosed by a simple polygonal
Jordan curve.  Its \emph{visibility kernel} is
\[
  K(F)
  =\{z\in F:[z,y]\subseteq F\text{ for every }y\in F\},
\]
and let $C(F)=\conv F$ be its convex hull.  We write $\area{E}$ for planar
Lebesgue area and adopt $\area{\varnothing}=0$.  We use Sibley's guard-point
ratio $G$ and write $A(F):=E(F)$ for his exterior area
ratio~\cite{Sibley}:
\[
  G(F)=\frac{\area{K(F)}}{\area{F}},
  \qquad
  A(F)=\frac{\area{F}}{\area{C(F)}}.
\]
Both equal one for a convex polygon.  Nakano proved Sibley's conjectured
inequality $G(F)\leq A(F)$ and, in the same work, refuted the analogous
perimeter comparison~\cite{Sibley,Nakano}.

Our main result strengthens the area comparison by an optimal factor.

\begin{theorem}[Kernel--hull deficit inequality]
\label{thm:C1-intro}
For every simple polygonal region $F$, with $K=K(F)$ and $C=C(F)$,
\begin{equation}
  \area{F}\bigl(\area{F}-\area{K}\bigr)
  \;\geq\;
  2\area{K}\bigl(\area{C}-\area{F}\bigr).
  \tag{C1}
  \label{eq:C1}
\end{equation}
\addtocounter{equation}{1}
Equivalently,
\begin{equation}
  G(F)\leq \frac{A(F)}{2-A(F)}.
  \label{eq:ratio-form}
\end{equation}
Empty kernels and nonempty kernels of area zero are included.
\end{theorem}

Since $0<A(F)\leq1$,
\[
  \frac{A(F)}{2-A(F)}\leq A(F),
\]
with equality only at $A(F)=1$.  Thus \eqref{eq:ratio-form} implies Nakano's
area inequality and is strictly stronger whenever $A(F)<1$.  The coefficient
$2$ in \eqref{eq:C1} cannot be increased: in
\Cref{sec:sharpness} we give a one-parameter family of nonconvex equality
examples.

The proof of \Cref{thm:C1-intro} is driven by a statement with no polygonal
boundary in its final hypotheses.  If $K$ is convex and $x$ is a point, call
$\conv(K\cup\{x\})$ the cap from $x$ over $K$.  When $K$ is a Euclidean ball,
related unions are called \emph{spiky balls}, and convex examples are called
\emph{cap bodies}, in the illumination literature~\cite{BezdekIvanovStrachan};
that adjacent literature studies illumination rather than the area inequality
below.

\begin{theorem}[Cap-union inequality]
\label{thm:S-intro}
Let $K\subset\R^2$ be nonempty, compact, and convex.  Let
$X\subset\R^2$ be finite and nonempty, put $C=\conv X$, and assume
$K\subseteq C$.  Define
\begin{equation}
  U=\bigcup_{x\in X}\conv(K\cup\{x\}).
  \label{eq:cap-union}
\end{equation}
Then
\begin{equation}
  \area{U}^{2}+\area{K}\area{U}
  \;\geq\;2\area{K}\area{C}.
  \tag{$S^\circ$}
  \label{eq:Szero}
\end{equation}
\addtocounter{equation}{1}
No positive-area, interior, convex-position, irredundancy, or
exterior-generator hypothesis is required.
\end{theorem}

There are three main steps.  First, cyclic sorting maximizes signed area among
closed polygonal walks with a fixed directed edge multiset.  Reversing the
order of the boundary edges between consecutive hull vertices therefore
produces a convex companion $H$ satisfying
\[
  \area{H}+\area{U}\geq2\area{C}.
\]
Second, the fact that every oriented boundary edge of $U$ has $K$ on its
inward side gives an anisotropic support bound.  The polygonal mixed-area
formula and Minkowski's inequality turn it into
\[
  \area{U}^{2}\geq\area{K}\area{H}.
\]
Combining the two inequalities proves \eqref{eq:Szero} for polygonal $K$.
Third, a fixed-inball inner approximation passes to arbitrary compact convex
$K$ without appealing to area continuity for arbitrary nonconvex unions.

The argument treats redundant and interior generators, coincident caps,
parallel boundary pieces, support faces, points and segments.  No division is
used in the final combination.  These details are important both
mathematically and in the formal verification described in
\Cref{sec:formal}.

\section{Definitions and convex-geometric preliminaries}
\label{sec:preliminaries}

\subsection{Polygons, support functions, and mixed area}

A \emph{simple polygonal region} is the compact Jordan domain bounded by a
finite simple polygonal cycle.  We orient its boundary counterclockwise.
Changing the orientation changes none of the sets or areas in our results.
For such a region $F$, its convex hull agrees with the convex hull of its
listed vertices.

The visibility kernel used here is automatically a compact convex set.  More
generally, let $R\subset\R^2$ be compact and define
\[
  K(R)=\{z\in R:[z,y]\subseteq R\text{ for every }y\in R\}.
\]
If $z_0,z_1\in K(R)$ and $y\in R$, then $[z_1,y]\subseteq R$; joining every
point of this segment to $z_0$ shows
$\conv\{z_0,z_1,y\}\subseteq R$.  Hence every point of $[z_0,z_1]$ sees
every $y\in R$, proving that $K(R)$ is convex.  It is closed as well: if
$z_n\in K(R)$ and $z_n\to z$, closedness of $R$ gives $z\in R$, and for
every fixed $y\in R$ and $t\in[0,1]$ the points
$(1-t)z_n+ty\in R$ converge to $(1-t)z+ty\in R$.  Thus $K(R)$ is a closed
subset of the compact set $R$.

For a nonempty compact convex set $K\subset\R^2$, its support function is
\[
  h_K(z)=\max_{y\in K}\inner{y}{z}.
\]
For nonempty compact convex $A,B\subset\R^2$, the mixed area $V(A,B)$ is
normalized by the planar Minkowski polynomial
\begin{equation}
  \area{A+tB}=\area{A}+2tV(A,B)+t^2\area{B},
  \qquad t\geq0.
  \label{eq:mixed-area-normalization}
\end{equation}
as in Schneider~\cite[\S5.1]{Schneider}.
In particular, $V(A,A)=\area{A}$ and $V$ is symmetric.  If $H$ is a convex
polygon with counterclockwise directed edge vectors $e_i$, define the outward
conormal
\[
  N(a,b)=(b,-a).
\]

\begin{lemma}[Polygonal support formula]
\label{lem:support-mixed}
Let $H\subset\R^2$ be a convex polygon with nonempty interior whose boundary is
traversed counterclockwise with directed edge vectors $e_1,\dots,e_m$, and let
$K\subset\R^2$ be a nonempty compact convex set.  Put $R=\operatorname{diam}K$.
Then for every $t\in(0,1]$
\begin{equation}
  \Bigl\lvert\,\area{H+tK}-\area{H}-t\sum_{i=1}^{m}h_K\bigl(N(e_i)\bigr)
  \,\Bigr\rvert\;\leq\;m(1+\pi)R^{2}\,t^{2}.
  \label{eq:support-collar}
\end{equation}
Consequently the right derivative of $t\mapsto\area{H+tK}$ at $t=0$ exists and
equals $\sum_i h_K(N(e_i))$, and therefore, by
\eqref{eq:mixed-area-normalization},
\begin{equation}
  2V(H,K)=\sum_{i=1}^{m}h_K\bigl(N(e_i)\bigr).
  \label{eq:support-mixed}
\end{equation}
No interior, smoothness or polygonality hypothesis on $K$ is used.
\end{lemma}

\begin{proof}
Both sides of \eqref{eq:support-mixed} are unchanged if $K$ is translated:
$\area{H+t(K+v)}=\area{H+tK}$, while $N$ is linear and the directed edges of a
closed polygon sum to $0$, so $\sum_iN(e_i)=N(\sum_ie_i)=0$ and
$\sum_ih_{K+v}(N(e_i))=\sum_ih_K(N(e_i))+\inner{v}{\sum_iN(e_i)}
=\sum_ih_K(N(e_i))$.  We may therefore assume $0\in K$; then $h_K\geq0$,
$H\subseteq H+tK$, and $\max_{y\in K}\lvert y\rvert\leq R$, so $K\subseteq RD$
with $D$ the closed unit disc.  Likewise, if an edge vector $e$ is split into
positive multiples $e'+e''=e$ of itself, then
$h_K(N(e'))+h_K(N(e''))=(\lvert e'\rvert+\lvert e''\rvert)h_K(N(e)/\lvert
e\rvert)=h_K(N(e))$ by positive homogeneity of $h_K$, so we may also assume the
edge list is reduced: no two consecutive $e_i$ are positive multiples of each
other.  Write $E_i$ for the closed $i$th edge, $v_1,\dots,v_m$ for the
vertices, $n_i=N(e_i)$, $u_i=n_i/\lvert e_i\rvert$ (well defined since edge vectors of a
polygon are nonzero; a unit vector, since
$\lvert N(w)\rvert=\lvert w\rvert$), and $c_i=h_H(u_i)$.  Counterclockwise
traversal makes $u_i$ the \emph{outward} unit normal of $E_i$, so
$E_i=\{z\in H:\inner{z}{u_i}=c_i\}$ and $h_K(n_i)=\lvert e_i\rvert h_K(u_i)$.

Let $\rho$ denote the nearest-point projection onto the closed convex set $H$,
characterised by $\inner{x-\rho(x)}{y-\rho(x)}\leq0$ for all $y\in H$.  For
$x\notin H$ we have $\rho(x)\in\partial H$, and $\partial H$ is the disjoint
union of the sets $\operatorname{relint}E_i$ and the vertices.  Hence, setting
\[
  S_i=\{x\notin H:\rho(x)\in\operatorname{relint}E_i\},\qquad
  W_j=\{x\notin H:\rho(x)=v_j\},
\]
the $2m$ Borel sets $S_i,W_j$ are pairwise disjoint with union $\R^2\setminus
H$.

\emph{Step 1: $S_i=\operatorname{relint}(E_i)+(0,\infty)u_i$.}  If $x\in S_i$
and $p=\rho(x)$, then $p\pm\eps e_i\in E_i\subseteq H$ for small $\eps>0$, so
the variational inequality gives $\inner{x-p}{e_i}=0$, whence $x-p=ru_i$ with
$r\neq0$.  Since $\operatorname{int}H\neq\emptyset$ and $p\in
\operatorname{relint}E_i$, we have $p-\eps u_i\in H$ for small $\eps>0$: pick
$z_0\in\operatorname{int}H$, so $\inner{z_0}{u_i}<c_i$ and
$p_s=(1-s)p+sz_0\in\operatorname{int}H$ for $s\in(0,1]$, and for $\eps$ small
$p-\eps u_i$ lies in the convex hull of some $p_s$ and a relative
neighbourhood of $p$ in $E_i$.  Testing the variational inequality at
$y=p-\eps u_i$ gives $-\eps r\leq0$, so $r>0$.  Conversely, if $p\in
\operatorname{relint}E_i$ and $r>0$ then $\inner{p+ru_i}{u_i}=c_i+r>c_i$, so
$p+ru_i\notin H$, and $\inner{(p+ru_i)-p}{y-p}=r(\inner{y}{u_i}-c_i)\leq0$ for
every $y\in H$, so $\rho(p+ru_i)=p$.

\emph{Step 2: the upper bound.}  Let $x\in(H+tK)\cap S_i$, say $x=p+ru_i$ as in
Step 1.  Since $h_{A+B}=h_A+h_B$,
$\inner{x}{u_i}\leq h_{H+tK}(u_i)=c_i+th_K(u_i)$, while
$\inner{x}{u_i}=c_i+r$; hence $0<r\leq th_K(u_i)$ and
\[
  (H+tK)\cap S_i\subseteq E_i+[0,th_K(u_i)]u_i,
  \qquad
  \area{(H+tK)\cap S_i}\leq t\lvert e_i\rvert h_K(u_i)=t\,h_K(n_i).
\]
From $K\subseteq RD$ we get $H+tK\subseteq H+tRD$, so every $x\in H+tK$
satisfies $\lvert x-\rho(x)\rvert\leq tR$; on $W_j$ this says $x$ lies in the
disc of radius $tR$ about $v_j$, whence $\area{(H+tK)\cap W_j}\leq\pi R^2t^2$.
As $H\subseteq H+tK$, additivity of Lebesgue measure over the partition gives
\[
  \area{H+tK}\;\leq\;\area{H}+t\sum_ih_K(n_i)+m\pi R^2t^2 .
\]

\emph{Step 3: the lower bound.}  Choose $y_i\in K$ with
$\inner{y_i}{u_i}=h_K(u_i)=:h_i\geq0$, possible since $K$ is compact and
nonempty, and write $y_i=s_i\tau_i+h_iu_i$ in the orthonormal frame
$\tau_i=e_i/\lvert e_i\rvert$, $u_i$; then $\lvert s_i\rvert\leq R$ and
$h_i\leq R$.  Put $P_i=E_i+[0,t]y_i$.  Then $P_i\subseteq H+tK$, because
$E_i\subseteq H$ and $\sigma y_i\in K$ for $\sigma\in[0,1]$ by convexity of $K$
and $0\in K$.  If $h_i=0$ the bound below is trivial, so assume $h_i>0$ and
parametrise $P_i$ by the injective affine map
$\Phi(\theta,\sigma)=q_i+\theta\tau_i+\sigma ty_i$ on
$[0,\lvert e_i\rvert]\times[0,1]$, $q_i$ the initial endpoint of $E_i$; its
Jacobian is $\lvert\det(\tau_i,ty_i)\rvert=th_i$, so
$\area{P_i}=t\lvert e_i\rvert h_i=t\,h_K(n_i)$.  Since
$\Phi(\theta,\sigma)=\bigl(q_i+(\theta+\sigma ts_i)\tau_i\bigr)+\sigma th_iu_i$,
Step 1 shows $\Phi(\theta,\sigma)\in S_i$ whenever $\sigma>0$ and
$\theta+\sigma ts_i\in(0,\lvert e_i\rvert)$.  For each fixed $\sigma$ the
excluded $\theta$'s form at most two intervals of total length
$\leq t\lvert s_i\rvert$, so the excluded parameter set has measure
$\leq t\lvert s_i\rvert$ and its image has area $\leq t\lvert
s_i\rvert\cdot th_i\leq R^2t^2$.  Hence $\area{P_i\cap S_i}\geq
t\,h_K(n_i)-R^2t^2$.  The sets $P_i\cap S_i$ lie in $H+tK$, are pairwise
disjoint and miss $H$, so
\[
  \area{H+tK}\;\geq\;\area{H}+t\sum_ih_K(n_i)-mR^2t^2 .
\]

The two bounds give \eqref{eq:support-collar}.  Dividing by $t$ and letting
$t\to0+$ shows that the right derivative of $t\mapsto\area{H+tK}$ at $t=0$
exists and equals $\sum_ih_K(N(e_i))$; by
\eqref{eq:mixed-area-normalization} that derivative equals $2V(H,K)$, which is
\eqref{eq:support-mixed}.
\end{proof}

\begin{remark}
\label{rem:support-mixed-riders}
\Cref{lem:support-mixed} needs no interior hypothesis on $K$: a nonempty
compact convex $K\subset\R^2$ with empty interior is a point or a segment, and
the proof used only that $K$ is nonempty, compact and convex.  It is also
insensitive to subdividing edges of $H$: if $e$ is replaced by positive
multiples $e'+e''=e$ of itself, then $h_K(N(e'))+h_K(N(e''))=h_K(N(e))$ by
linearity of $N$ and positive homogeneity of $h_K$, so \eqref{eq:support-mixed}
for a reduced edge list gives it for every subdivision.  Finally, translating
$K$ causes no ambiguity: $\sum_iN(e_i)=N(\sum_ie_i)=0$, so
$\sum_ih_{K+v}(N(e_i))=\sum_ih_K(N(e_i))$, matching $V(H,K+v)=V(H,K)$.
\end{remark}

For the general theory of mixed volumes and their degeneracies see
Schneider~\cite[Chapters~5 and~7]{Schneider}; see also Shenfeld--van
Handel~\cite{ShenfeldVanHandel} for the extremals of Minkowski's quadratic
inequality, whose results are stated for dimension $n\geq3$ and are therefore
not available in the plane.

We will use the planar Minkowski inequality in precisely this normalization.

\begin{lemma}[Minkowski's mixed-area inequality]
\label{lem:minkowski}
For nonempty compact convex $A,B\subset\R^2$,
\begin{equation}
  V(A,B)^2\geq\area{A}\area{B}.
  \label{eq:minkowski}
\end{equation}
No nonempty-interior hypothesis is needed.
\end{lemma}

\begin{proof}
The planar Brunn--Minkowski inequality
\cite[\S7.1]{Schneider} gives
\[
  \area{A+tB}^{1/2}
  \geq \area{A}^{1/2}+t\area{B}^{1/2}
  \qquad(t\geq0).
\]
After squaring and comparing with
\eqref{eq:mixed-area-normalization}, for every $t>0$ we obtain
$V(A,B)\geq\sqrt{\area{A}\area{B}}$.  Squaring yields
\eqref{eq:minkowski}.  If either area vanishes, the same calculation simply
has zero right-hand side; no limiting nondegeneracy argument is required.
\end{proof}

\subsection{Signed area and sorted companions}

For an oriented polygonal chain
$\gamma=(q_0,q_1,\dots,q_m)$, define its line-integral area by
\begin{equation}
  I(\gamma)=\frac12\sum_{i=0}^{m-1}\det(q_i,q_{i+1}).
  \label{eq:chain-area}
\end{equation}
For a closed walk $Q$ we write $A_s(Q)=I(Q)$.  This signed area is defined
even when $Q$ self-intersects.  For the positive boundary of a simple
polygonal region, it is the ordinary area.

Let $Q$ be any finite closed polygonal walk, and delete its zero directed
edges.  The remaining edge vectors $e_1,\ldots,e_m$ sum to zero.  Put them in
cyclic nondecreasing polar-angle order and concatenate them.  Their partial
sums trace the boundary of a compact convex polygonal set, unique up to
translation and the consolidation of codirected edges; the set may be a point
or segment if all edges are line-supported.  We call it the \emph{sorted
companion} and denote it by $\Sor(Q)$.  Relative to any displayed edge, the
determinants with later sorted edges first have one sign and then the other;
since their total is zero, all partial sums lie in the inward half-plane of
that edge.  This proves convexity directly.  The companion has positive area
whenever its edge directions are not supported on a line.

\section{Edge rearrangement and the companion inequality}
\label{sec:rearrangement}

The factor two in the main theorem ultimately comes from reversing an edge
block without reversing any of its vectors.

\begin{lemma}[Endpoint-preserving reversal]
\label{lem:reversal}
Let $\gamma=(q_0,\ldots,q_m)$ run from $a=q_0$ to $b=q_m$, and set
\[
  q_i^{\dagger}=a+b-q_{m-i}
  \qquad(0\leq i\leq m).
\]
Then $\gamma^\dagger=(q_0^\dagger,\ldots,q_m^\dagger)$ also runs from $a$ to
$b$, its directed edge vectors are those of $\gamma$ in reverse order, and
\begin{equation}
  I(\gamma^\dagger)=\det(a,b)-I(\gamma).
  \label{eq:reversal}
\end{equation}
\end{lemma}

\begin{proof}
The $i$th edge of $\gamma^\dagger$ is
\[
 (a+b-q_{m-i-1})-(a+b-q_{m-i})=q_{m-i}-q_{m-i-1},
\]
so the edge vectors occur in reverse order with their directions unchanged.
Expanding \eqref{eq:chain-area}, or applying the half-turn
$z\mapsto a+b-z$ to the line integral, gives \eqref{eq:reversal}.
No simplicity assumption is involved.
\end{proof}

We next record the rearrangement statement in the generality needed for the
virtual walk below.

\emph{Attribution.}  For \emph{simple} polygons the content of
\eqref{eq:sort-max} below is classical, and is half a century old under the
name \emph{flipturn}.  A flipturn replaces the pocket arc cut off by an exposed
pair of vertices by its half-turn about the midpoint of that pair; it
therefore leaves the multiset of directed edge vectors unchanged and alters
only their cyclic order.  Gr\"unbaum reports that Joss and Shannon proved, in
1973 or 1974, that a sequence of flipturns convexifies any simple polygon, by
the observation that each flipturn strictly increases the area, so no cyclic
order can recur; the result is reported with proof by
Gr\"unbaum~\cite[Theorem~3]{Grunbaum}.  Aichholzer et
al.~\cite[\S6]{Aichholzer} record explicitly that, because flipturns change
neither the directions nor the lengths of the edges, ``the final convex shape
of $P$ is the same for any convexifying flipturn sequence'' and can be
computed ``by sorting the edges of $P$ by their orientation''.  Codirected
edges may be listed in any order, since reordering them only subdivides a
single side; the resulting shape is $\Sor(Q)$, with the tie convention used
here.  Together these yield
\eqref{eq:sort-max} for simple polygons.  The proof given below does not
depend on that literature.

The flipturn argument depends on the existence of an exposed pair and thus on
simplicity, and Gr\"unbaum and Aichholzer et al.\ state their results for
simple polygons only.  \emph{Beyond the simple case the statement is also
known, and has been for four decades.}  B\"or\"oczky, B\'ar\'any, Makai and
Pach~\cite[Theorem~2$''$]{BBMP} prove, for affine cycles with integer
coefficients whose nondegenerate normals span the plane, that the signed area
is at most the area of the convexification, with equality exactly when the
cycle is a well-oriented traversal of the convexified boundary; in the plane
that is \eqref{eq:sort-max} together with its equality case, for
self-intersecting walks apart from the immediate line-supported degeneracy.
They report a planar
antecedent in F\'ary and Makai~\cite[Lemma~1]{FaryMakai}.  Building on
both, Charon--Pierron~\cite[Lemma~4.3 and Theorem~4.4]{CharonPierron} prove the
statement for closed curves without a simplicity hypothesis, in the
fixed-directed-length-measure setting, together with the characterisation of
the maximisers, excluding measures supported on two antipodal directions.
In that excluded line-supported case both the signed area and the companion
area are zero, so \eqref{eq:sort-max} is immediate.  We therefore claim no
novelty for \Cref{lem:sort-max}; the
finite polygonal proof is included only because it is short, self-contained,
and phrased in the edge-multiset language used throughout this paper.

\begin{lemma}[Cyclic edge sorting maximizes signed area]
\label{lem:sort-max}
Let $Q$ be any finite closed polygonal walk, possibly self-intersecting.  Then
\begin{equation}
  A_s(Q)\leq\area{\Sor(Q)}.
  \label{eq:sort-max}
\end{equation}
Zero edges may be deleted, and codirected edges may be kept separately or
consolidated.
\end{lemma}

\begin{proof}
We first suppose that all vertices and crossings of $Q$ are in general
position: proper crossings are transverse, no three edge interiors meet, and
there are no overlapping collinear edge intervals, and no walk vertex lies in
the interior of a nonincident edge.  Split every edge at every crossing,
including endpoint contacts.  The resulting finite planar directed multigraph is Eulerian.
Following unused directed edges and removing the first cycle whenever a vertex
repeats decomposes the split edge multiset into directed graph-simple cycles
$\Gamma_1,\ldots,\Gamma_r$.  Because the graph has been planarized, each
$\Gamma_j$ is a geometric Jordan polygon, and additivity over directed edges
gives
\begin{equation}
  A_s(Q)=\sum_{j=1}^r A_s(\Gamma_j).
  \label{eq:cycle-decomposition}
\end{equation}

Let $H_j$ be the sorted companion of $\Gamma_j$.  If $\Gamma_j$ is clockwise,
then $A_s(\Gamma_j)<0\leq\area{H_j}$.  Suppose it is counterclockwise.  After
an arbitrarily small generic perturbation inside the zero-sum edge-list
space, every labeled edge order is free of accidental collinearities.  By the
continuity of $A_s$ for a fixed edge order and of the sorted-companion area,
both established in the closing paragraph of this proof, the inequality for
the perturbed edge list passes to $\Gamma_j$ itself; we may therefore argue
for the perturbed polygon.  If the
polygon is nonconvex, choose a positive-area pocket between a hull lid
$[a,b]$ and the corresponding boundary chain.  Replace that chain by the
endpoint-preserving reversal from \Cref{lem:reversal}.  Genericity leaves $a$
and $b$ as the only boundary points on the lid line; the new chain then lies
on the far side of the hull lid, meeting the line only at its endpoints,
while the unchanged boundary lies on the near side, so the new polygon is
again simple.  Formula \eqref{eq:reversal} says
that its area increases by twice the pocket area.  Only the order of a
consecutive labeled edge block changes.  There are finitely many labeled
cyclic edge orders, so strict increase prevents recurrence and the procedure
terminates.  A terminal polygon has no pocket and is convex; its positive
boundary edges are in cyclic angular order.  To see explicitly that this
identifies the terminal polygon, fix one boundary vertex and list the ordered
edges as $e_1,\ldots,e_m$.  Every subsequent vertex is then the corresponding
partial sum $e_1+\cdots+e_k$.  Thus a based polygon is determined by its
ordered edge list; changing the initial edge only cyclically rotates the list,
and changing the order of codirected edges only subdivides the same straight
side.  Consequently every convex realization of this edge multiset is
its sorted companion up to translation.  Hence
\begin{equation}
  A_s(\Gamma_j)\leq\area{H_j}.
  \label{eq:cycle-sort}
\end{equation}

Splitting a vector into positive codirected pieces merely subdivides a side
of its sorted companion.  Merging all sorted edge lists in angular order is
the usual edge description of polygonal Minkowski addition: support functions
add under Minkowski sum, a planar convex body is determined by its support
function, and the convex polygon whose edge vectors are the angular merge has
support function the sum of the summands' support functions.  Segment
summands and antiparallel cross-summand pairs are covered, the former
contributing a single edge direction and the latter remaining distinct in the
cyclic angular order.  Hence, up to translation,
\begin{equation}
  \Sor(Q)=H_1+\cdots+H_r.
  \label{eq:sorted-minkowski-sum}
\end{equation}
For compact convex $A,B\subset\R^2$ one has
$\area{A+B}\geq\area{A}+\area{B}$.  Indeed, choose a unit vector $n$, a
support point $a_+\in A$ maximizing $\inner{a}{n}$, and a point
$b_-\in B$ minimizing $\inner{b}{n}$.  The translates $A+b_-$ and $a_++B$
both lie in $A+B$ and have interiors in opposite half-planes separated by the
line
\[
  \inner{x}{n}=h_A(n)+\min_{b\in B}\inner{b}{n}.
\]
Their intersection has area zero, proving the claimed superadditivity.
Induction in \eqref{eq:sorted-minkowski-sum}, together with
\eqref{eq:cycle-decomposition} and \eqref{eq:cycle-sort}, proves
\eqref{eq:sort-max} in general position.

If at most two edges of $Q$ are nonzero they are opposite, so $A_s(Q)=0$ while
the companion is a point or a segment of zero area, and \eqref{eq:sort-max} is
immediate; assume from now on that at least three are nonzero.
For an arbitrary closed walk, perturb its vertices, keeping the endpoint
closed, to general-position walks $Q_{\eps}\to Q$.  The required continuity
can be read directly from the edges.  If $e_1,\ldots,e_m$ is any ordered
closed edge list, based at the origin, then
\[
  A_s(e_1,\ldots,e_m)
  =\frac12\sum_{1\leq p<q\leq m}\det(e_p,e_q).
\]
Indeed, the initial vertex of the $q$th edge is
\(\sum_{p<q}e_p\), and substitution into the shoelace sum gives the formula.
For any fixed edge order this is polynomial in the edge coordinates.
Replacing an adjacent order $a,b$ by $b,a$ changes $A_s$ by
$-\det(a,b)$, so the two chamber formulas agree on a codirected angular wall.
Simultaneous ties are resolved by adjacent swaps within the tied block.
Indeed, any two orders compatible with the limiting angles differ by adjacent
transpositions of edges that are codirected at the limit, so all chamber
formulas whose chambers approach a given closed edge list take the same value
at it; since the chambers are finite in number and each formula is
polynomial, this agreement on shared strata is precisely continuity of the
glued function.
Antiparallel directions remain distinct in the cyclic angular order and cause
no swap.  If an edge $e$ tends to zero, every determinant term involving
$e$ tends to zero, independently of its temporary sorted position.  Finally,
cyclic rotation leaves the formula unchanged because $\sum_p e_p=0$; hence a
direction crossing the angular branch cut creates no jump.  The
sorted-companion area is therefore continuous on the space of closed edge
lists.  Passing to the limit gives \eqref{eq:sort-max}, including parallel,
repeated, antiparallel, zero, and multiply intersecting cases.
\end{proof}

We use the following elementary polygonal-Jordan facts.  If $\Gamma$ is a
simple closed polygonal curve, then $\R^2\setminus\Gamma$ has exactly two open
components; write $\intr\Gamma$ for the bounded one and $\extr\Gamma$ for the
unbounded one.  They satisfy
$\partial(\intr\Gamma)=\partial(\extr\Gamma)=\Gamma$ and
$\operatorname{cl}(\intr\Gamma)=\intr\Gamma\cup\Gamma$.  For a polygonal
region $F$, $\intr F$ retains its ordinary topological meaning.  At a point in
the relative interior of a single edge, a sufficiently small disc is divided
by that edge into one half-disc in each component.  For a closed polygonal walk
$W$, let $\Ind_W$ denote its signed-angle winding number.  It is additive over
directed edges and invariant under subdivision; for a simple $\Gamma$ it is
zero on $\extr\Gamma$ and one fixed value in $\{1,-1\}$ on $\intr\Gamma$,
while crossing a directed edge locally from right to left raises it by one.
Consequently a simple polygonal boundary with $A_s>0$ has index one on its
interior and its Jordan domain lies locally to the left of every directed
boundary edge; we call this the positive, or counterclockwise, orientation.
We also use that the complement of a compact convex planar set is
path-connected.  When two polygonal arcs with the same endpoints are otherwise
disjoint, their union is understood with the induced finite simple
polygonal-cycle presentation.

\begin{lemma}[Hull polygon]
\label{lem:hull-polygon}
Let $S\subset\R^2$ be finite with $C\defeq\conv S$ of nonempty interior.  Then
the set $E$ of extreme points of $C$ is finite, $E\subseteq S$,
$r\defeq\lvert E\rvert\geq3$, and $C=\conv E$.  Fix $o\in\intr C$ and
translate so that $o=0$; the arguments of the points of $E$ are pairwise
distinct, and if $p_0,\dots,p_{r-1}$ are listed by increasing argument then
the closed cycle $(p_0,\dots,p_{r-1})$ is a simple closed polygonal cycle with
image $\partial C$ and with $A_s>0$.  In particular, this cyclic order is
independent of the choice of $o$, up to cyclic rotation.  Each
$L_j=[p_j,p_{j+1}]$ lies on a supporting line of $C$, the sets
$\relint L_j$ are exactly the connected components of
$\partial C\setminus E$, and $N(p_{j+1}-p_j)$ is the outward normal direction
of $C$ along $L_j$.
\end{lemma}

\begin{proof}
Let $x\in\conv S$ be extreme and choose a convex representation
$x=\sum_i\lambda_i s_i$.  If only one coefficient is positive, then
$x\in S$.  Otherwise choose $0<\lambda_1<1$ and put
$y=(1-\lambda_1)^{-1}\sum_{i>1}\lambda_i s_i\in C$, so that
$x=\lambda_1s_1+(1-\lambda_1)y$.  If $s_1=y$, then $x=s_1\in S$ and we are
done.  Otherwise $[s_1,y]$ is a nondegenerate segment in $C$ with $x$ in its
relative interior, contradicting extremality.  Thus $E\subseteq S$.
If $s\in S$ is not extreme, write $s=(1-t)u+tv$ with $u\neq v$ in $C$ and
expand $u,v$ over $S$: the resulting coefficient of $s$ is less than one, so
$s\in\conv(S\setminus\{s\})$ and $s$ may be deleted without changing the hull.
Iterating gives $C=\conv E$; and the convex hull of at most two points has
empty interior, so $r\geq3$.  An extreme point is not interior, so
$E\subset\partial C$ and $0\notin E$; two extreme points on one ray from $0$
would put the nearer in the relative interior of the segment joining $0$ to
the farther, so their arguments are distinct.

Because $0\in\intr C$, for each $\theta$ the set
$\{t\geq0:te^{i\theta}\in C\}$ is an interval $[0,r_C(\theta)]$ with
$r_C(\theta)>0$, and $te^{i\theta}\in\intr C$ for $t<r_C(\theta)$; hence
$\partial C=\{r_C(\theta)e^{i\theta}\}$.  No cyclic gap between consecutive
arguments is at least $\pi$, since otherwise $E$, hence $C$, would lie in a
closed half-plane through $0$.  Therefore
$\det(p_j,p_{j+1})>0$ for every $j$, so $A_s>0$ and
$0\notin\operatorname{line}(p_j,p_{j+1})$.  A segment whose endpoints subtend
an angle less than $\pi$ at $0$ and whose line misses $0$ meets each ray from
$0$ in at most one point, and exactly one for arguments in the closed interval
it subtends.  Concatenating, the cycle is a radial graph over the full circle,
hence simple, and its image is
$\{\rho(\theta)e^{i\theta}\}$ with $\rho\leq r_C$ and
$\rho(\theta_j)=r_C(\theta_j)$.

Suppose $\rho(\theta)<r_C(\theta)$ for some
$\theta\in(\theta_j,\theta_{j+1})$, and put
$y=r_C(\theta)e^{i\theta}\in\partial C$.  Approximating $y$ from outside $C$
and taking nearest points produces a unit vector $n$ with
$\inner{z-y}{n}\leq0$ for all $z\in C$.  Thus
$h\defeq\max_{z\in C}\inner{z}{n}=\inner{y}{n}>0$, and
$C\cap\{\inner{\cdot}{n}=h\}=\conv E_n$ with
$E_n=\{p\in E:\inner{p}{n}=h\}$, since a convex combination of points of $E$
attains the maximum only if supported on $E_n$.  If $E_n$ is a single point
then $y\in E$ and $\theta$ is one of the $\theta_k$, a contradiction.
Otherwise $\conv E_n$ is a segment on a line missing $0$, subtending an
angular interval that contains $\theta$ and hence, since no $\theta_k$ lies in
$(\theta_j,\theta_{j+1})$, contains
$[\theta_j,\theta_{j+1}]$.  Were $\theta_j$ interior to that interval, the ray
at $\theta_j$ would meet $\conv E_n\subseteq\partial C$ at a point of modulus
$r_C(\theta_j)=\lvert p_j\rvert$, namely $p_j$, exhibiting $p_j$ in the
relative interior of a segment of $C$.  Likewise $\theta_{j+1}$ cannot be
interior to that angular interval, by the same argument at $p_{j+1}$.  Hence
$\conv E_n=[p_j,p_{j+1}]$ and $y$ lies on the cycle, contradicting
$\rho(\theta)<r_C(\theta)$.  So $\rho=r_C$ and the image is $\partial C$; the
same computation exhibits each $[p_j,p_{j+1}]$ as a face.  Finally, $C$ is a
positively oriented simple polygonal region, so the local-left orientation
fact shows that $N(p_{j+1}-p_j)$ points out of $C$.
\end{proof}

\begin{lemma}[Crosscuts of a convex polygon]
\label{lem:crosscut}
Let $C$ be as in \Cref{lem:hull-polygon}, with $\partial C$ carrying its
counterclockwise cycle $P$.  Call a polygonal arc $A_0\subseteq C$ with
distinct endpoints $a_0,c_0\in\partial C$ and
$A_0\setminus\{a_0,c_0\}\subseteq\intr C$ a \emph{crosscut}.  Write
$\sigma_R$ for the open arc of $\partial C$ from $a_0$ counterclockwise to
$c_0$ and $\sigma_L$ for the other one, and put
$J_L=A_0\cup\operatorname{cl}\sigma_L$ and
$J_R=A_0\cup\operatorname{cl}\sigma_R$.  Then:
\begin{enumerate}[label=\textup{(\roman*)}]
  \item $J_L,J_R$ are simple closed polygonal curves with
        $\intr J_L,\intr J_R\subseteq\intr C$, and
        $\sigma_R\subseteq\extr J_L$, $\sigma_L\subseteq\extr J_R$;
  \item $\intr C\setminus A_0=\intr J_L\sqcup\intr J_R$.
\end{enumerate}
\end{lemma}

\begin{proof}
(i) The intersection
$A_0\cap\operatorname{cl}\sigma_L=\{a_0,c_0\}$ because
$A_0\setminus\{a_0,c_0\}\subseteq\intr C$; two arcs with common endpoints and
otherwise disjoint form a simple closed curve.  After inserting their
endpoints as vertices and consolidating codirected pieces, it has a finite
simple polygonal-cycle presentation.  The path-connected set
$\R^2\setminus C$ is unbounded and misses
$J_L$, so it lies in $\extr J_L$ and $\intr J_L\subseteq\intr C$.  Then
$\sigma_R$, being connected, disjoint from $J_L$ and contained in
$\partial C$, lies in $\extr J_L$.  The assertions with $L$ and $R$
interchanged follow identically.

(ii) Orient $J_L$ as $A_0$ from $a_0$ to $c_0$ followed by $\sigma_L$, and
$J_R$ as $\sigma_R$ followed by $A_0$ reversed.  The two copies of $A_0$
contribute opposite signed angles, so
$\Ind_{J_L}+\Ind_{J_R}=\Ind_P$ off $A_0\cup\partial C$; and
$\Ind_P=1$ on $\intr C$.  Hence no point of
$\intr C\setminus A_0$ has both indices zero, and a nonzero index places it
inside one of the two curves.  Conversely, (i) puts both interiors in
$\intr C\setminus A_0$.  Finally, $\intr J_R$ is connected and misses $J_L$,
and every point of $\sigma_R$ is a limit of points of $\intr J_R$ by the
common-boundary fact.  Since $\sigma_R\subseteq\extr J_L$, it follows that
$\intr J_R\subseteq\extr J_L$, proving disjointness.
\end{proof}

\begin{lemma}[Hull order]
\label{lem:hull-order}
Let $F$ be a positively oriented simple polygonal region and $C_F=\conv F$.
Let $p_0,\dots,p_{r-1}$ be the extreme points of $C_F$ in the counterclockwise
order of $\partial C_F$ \textup{(\Cref{lem:hull-polygon}; $r\geq3$)}, and
write $E=\{p_0,\dots,p_{r-1}\}$.  Then
each $p_j$ is a vertex of $F$, and the counterclockwise traversal of
$\partial F$ meets $p_0,\dots,p_{r-1}$ in that cyclic order.  Consequently
$\partial F$ splits into chains $\gamma_j$ from $p_j$ to $p_{j+1}$ with
$\area{F}=\sum_j I(\gamma_j)$ and
$\area{C_F}=\sum_j I([p_j,p_{j+1}])$.
\end{lemma}

\begin{proof}
$C_F$ is the hull of the vertex list of $F$, so by
\Cref{lem:hull-polygon} its extreme points are among those vertices; in
particular $E\subseteq F$.

\emph{Step 1 (oriented crosscut).}  Let
$A_0\subseteq\partial F$ be a crosscut of $C_F$ traversed in the positive
direction from $a_0$ to $c_0$.  Its relative interior contains a point $p$
interior to an edge of $\partial F$ and not a vertex of $A_0$, because it is a
nonempty polygonal open arc with only finitely many vertices.  Choose $\eps>0$
smaller than the local two-sidedness and unit-jump radii for $\partial F$,
$J_L$, and $J_R$, and so small that
$D(p,\eps)\cap\partial F=D(p,\eps)\cap\operatorname{line}(e)$,
$D(p,\eps)\subseteq\intr C_F$, and
$D(p,\eps)\cap\partial F\subseteq A_0$.  The disc then meets $J_L$ and $J_R$
in that same diameter.  Write $D_L,D_R$ for the half-discs left and right of
the traversal.  By the local-left orientation fact
$D_L\subseteq\intr F$.  Let $x,y$ be the values of $\Ind_{J_L}$ on
$D_L,D_R$, and $s,t$ those of $\Ind_{J_R}$.  Crossing an edge from right to
left raises the index by one, so $x-y=1$ and $t-s=1$; the latter uses that
$A_0$ is traversed backwards in $J_R$.  As in the proof of
\Cref{lem:crosscut}(ii), the oppositely traversed copies of $A_0$ cancel, so
$\Ind_{J_L}+\Ind_{J_R}=\Ind_P=1$ on $\intr C_F\setminus A_0$ and
$x+s=y+t=1$.  Each simple curve's index takes
only zero and one fixed sign in $\{\pm1\}$, so $x-y=1$ leaves
$(x,y)\in\{(1,0),(0,-1)\}$.  The second case forces $s=1$ and then $t=2$,
which is impossible.  Hence $\Ind_{J_L}=1$ on $D_L$, and
$D_L\subseteq\intr J_L$.

\emph{Step 2 (no part of $F$ on the right).}  The connected set $\intr F$
misses $J_L$ and meets $\intr J_L$ by Step~1, so
$\intr F\subseteq\intr J_L$.  Regular closedness gives
$F=\operatorname{cl}(\intr F)\subseteq\intr J_L\cup J_L$, which misses
$\sigma_R$ by \Cref{lem:crosscut}(i).  Thus $F\cap\sigma_R=\emptyset$.
Since $E\subseteq F$, the connected arc $\sigma_R$ lies in
$\partial C_F\setminus E$, whose connected components are the relative
interiors of the hull edges by \Cref{lem:hull-polygon}.  It therefore lies in
the relative interior of one hull edge.  Thus a positively traversed excursion
of $\partial F$ into $\intr C_F$ re-lands further counterclockwise and without
passing an extreme point of $C_F$.

\emph{Step 3 (assembly).}  If $\partial F\subseteq\partial C_F$, then
$\partial F=\partial C_F$: otherwise removing a point of
$\partial C_F\setminus\partial F$ would embed the simple closed curve
$\partial F$ in an open arc.  Since both curves are positively oriented, the
claim is then immediate.  Otherwise parametrize $\partial F$ positively by
$g:[0,\Lambda]\to\partial F$, injective on $[0,\Lambda)$, with
$g(0)=g(\Lambda)=p_0$.  Put $T=g^{-1}(\partial C_F)$.  This is a finite union
of closed intervals and points, because
$\partial F\cap\partial C_F$ is a finite union of points and segments.

Each gap $(t_i,t_i')$ of $T$ is a genuine positively traversed crosscut.  Its
endpoints are distinct: equality could occur, by injectivity, only if the gap
were all of $(0,\Lambda)$, whereas the other $r-1\geq2$ extreme points give
contacts in $T\cap(0,\Lambda)$.  Step~2 therefore supplies a counterclockwise
hull arc from $g(t_i)$ to $g(t_i')$ of length
$0<\delta_i<\ell$, contained in one hull edge and with no extreme point in
its relative interior, where $\ell$ is the perimeter of $C_F$.  On every
nondegenerate component of $T$, decompose each hull-contact segment into
nondegenerate subsegments of individual edges of $F$.  The local-left orientation fact
forces each such subsegment to be codirected with the counterclockwise
direction of its containing hull edge.  Thus the counterclockwise arclength
modulo $\ell$, based at $p_0$, has a strictly increasing lift on every such
component; isolated contacts contribute nothing.  Concatenating these lifts
and increasing linearly by $\delta_i$ across each gap gives a continuous
nondecreasing lift $\Phi:[0,\Lambda]\to\R$ with $\Phi(0)=0$, congruent modulo
$\ell$ to the counterclockwise arclength of $g(t)$ for every $t\in T$.

Because $g(\Lambda)=g(0)$, $\Phi(\Lambda)$ is a nonnegative integral multiple
of $\ell$.  It is nonzero: otherwise monotonicity would make $\Phi$ identically
zero, while $g(T)$ contains all $r\geq3$ extreme points.  It cannot be at least
$2\ell$.  If it were, continuity would give $t_*\in(0,\Lambda)$ with
$\Phi(t_*)=\ell$.  If $t_*\in T$, congruence gives $g(t_*)=p_0$, contradicting
injectivity.  Otherwise $t_*\in(t_i,t_i')$ for a gap.  Neither endpoint can
have value $\ell$ by the just-handled $T$ case, so monotonicity and the affine
definition on the gap give strictly
$\Phi(t_i)<\ell<\Phi(t_i')$.  Since $\delta_i<\ell$, this places $p_0$
strictly inside the counterclockwise hull arc from $g(t_i)$ to $g(t_i')$,
contradicting Step~2.  Therefore $\Phi(\Lambda)=\ell$.  The extreme points,
whose lift values are their counterclockwise distances from $p_0$, are met in
counterclockwise order.

The two asserted area identities follow by cutting the directed edge list of
$\partial F$ at the $p_j$ and applying \Cref{lem:hull-polygon}.
\end{proof}

The next result converts the rearrangement lemma into the precise hull-area
estimate used later.  Its weak form $\area{\conv F}\leq\area{\Sor(F)}$,
valid for arbitrary closed polygons in the plane, is due to Pach and is
recorded as Proposition~1 of B\"or\"oczky, B\'ar\'any, Makai~Jr.
and Pach~\cite{BBMP}, where its cases of equality are also discussed.  The
lemma below strengthens it, on the smaller domain of positively oriented
simple regions, by the hull-deficit term $\area{\conv F}-\area{F}$; that
term is the source of the factor two in \Cref{thm:C1-intro}.  It is also
exactly the singleton-cycle specialization of Theorem~2 in Siegel's
author-hosted manuscript~\cite[Theorem~2]{Siegel}.  We retain the independent
proof because that manuscript is unpublished and because the hull-block
construction is used later.  Thus no novelty is claimed for \eqref{eq:CVX};
the contribution of C1 is the visibility-kernel and mixed-area synthesis
built around this exact prior ingredient.

\begin{lemma}[Sorted-companion inequality]
\label{lem:companion}
Let $F$ be a positively oriented simple polygonal region, let
$C_F=\conv F$, and let $H=\Sor(F)$.  Then
\begin{equation}
  \area{H}+\area{F}\geq2\area{C_F}.
  \tag{CVX}
  \label{eq:CVX}
\end{equation}
\addtocounter{equation}{1}
\end{lemma}

\begin{proof}
List the extreme points $p_0,\ldots,p_{r-1}$ of $C_F$ in counterclockwise
order; by \Cref{lem:hull-polygon} they are vertices of $F$ and $r\geq3$.  By
\Cref{lem:hull-order} the counterclockwise traversal of $\partial F$ meets them
in that cyclic order.

Let $\gamma_j$ be the full boundary chain of $F$ from $p_j$ to $p_{j+1}$,
with indices read cyclically, and let $L_j=[p_j,p_{j+1}]$.  Shoelace gives
\begin{equation}
  \area{F}=\sum_j I(\gamma_j),
  \qquad
  \area{C_F}=\sum_j I(L_j).
  \label{eq:block-shoelace}
\end{equation}
Replace every $\gamma_j$ virtually by its endpoint-preserving reversal and
concatenate the reversed blocks.  The resulting closed walk $Q$ need not be
simple, but it has exactly the directed edge multiset of $\partial F$.
By \Cref{lem:reversal} and \eqref{eq:block-shoelace},
\[
  A_s(Q)
  =\sum_j\bigl(2I(L_j)-I(\gamma_j)\bigr)
  =2\area{C_F}-\area{F}.
\]
Its sorted companion is $H$.  Applying \Cref{lem:sort-max} to $Q$ yields
$\area{H}\geq2\area{C_F}-\area{F}$, which is \eqref{eq:CVX}.
The signed calculation includes collinear hull contacts and zero-area
pockets; no perturbation of the hull blocks is needed.
\end{proof}

\section{Polygonal cap unions}
\label{sec:polygonal-caps}

We now prove the cap-union inequality when $K$ is a positive-area convex
polygon.

\begin{lemma}[Polygonal cap-union interface]
\label{lem:cap-interface}
Suppose that $K\subset\R^2$ is a positive-area compact convex polygon,
$X\subset\R^2$ is finite and nonempty, and
$K\subseteq C=\conv X$.  Let $U$ be defined by \eqref{eq:cap-union}.  Then:
\begin{enumerate}[label=\textup{(\roman*)}]
  \item $U$ is a positive-area compact simple polygonal region;
  \item $K\subseteq K(U)$, so every point of $K$ sees all of $U$;
  \item with positive boundary orientation, $K$ lies in the closed inward
        half-plane of every boundary edge of $U$; and
  \item $\conv U=C$.
\end{enumerate}
Interior and redundant generators, coincident cap edges, and support-face
contacts are allowed.
\end{lemma}

\begin{proof}
Choose $o\in\operatorname{int}K$.  Every cap
$P_x=\conv(K\cup\{x\})$ contains a fixed closed ball about $o$.  Its
intersection with the ray $o+\R_{\geq0}\theta$ is an interval
$[o,o+\rho_x(\theta)\theta]$, where the positive radial function $\rho_x$ is
continuous.  Since $X$ is finite,
\[
  \rho(\theta)=\max_{x\in X}\rho_x(\theta)
\]
is positive and continuous, and the ray section of $U$ is the interval with
endpoint $o+\rho(\theta)\theta$.  Thus this radial graph parametrizes
$\partial U$ as a Jordan curve.

Each $P_x$ is polygonal, and a point of $\partial U$ cannot lie in the
interior of any $P_x$.  Hence
$\partial U\subseteq\bigcup_x\partial P_x$.  Subdivide this finite segment
arrangement at all endpoints, crossings, and endpoints of collinear overlaps.
The Jordan boundary is a finite union of the resulting elementary segments,
and therefore is polygonal.  Since $K\subseteq U$, it encloses positive area.

If $z\in K$ and $y\in U$, choose $x$ with $y\in P_x$.  Convexity gives
$[z,y]\subseteq P_x\subseteq U$, proving (ii).  A boundary segment of $U$
lies on the boundary of at least one cap; that cap, and hence $K$, lies on its
occupied side, which is the inward side in the positive traversal.  This
proves (iii), including coincident cap edges because all caps share the same
positive-area $K$.

Finally, $K\cup X\subseteq C$ gives $U\subseteq C$, while
$x\in P_x\subseteq U$ for every $x\in X$.  Thus
$C=\conv X\subseteq\conv U\subseteq C$.
\end{proof}

For a positively oriented polygonal region $F$ with directed boundary edges
$e$, define the anisotropic support sum
\begin{equation}
  \Ssup(F)=\sum_e h_K\bigl(N(e)\bigr).
  \label{eq:anisotropic-support}
\end{equation}

This is an \emph{absolute} quantity, homogeneous of degree two in the data,
and it coincides with the anisotropic perimeter $\Per_K(F)$ of
Nakano~\cite{Nakano}.  We warn the reader that he reserves the symbol $P_K$
for the normalized ratio $P_K(F)=\Per_K(K)/\Per_K(F)$, which he defines only
when $K(F)$ has positive area; the two should not be conflated.

\begin{lemma}[Support and mixed-area bridge]
\label{lem:support-bridge}
Let $F$ be a positive-area simple polygonal region, and suppose the nonempty
compact convex set $K$ lies in every closed inward boundary-edge half-plane
of $F$.  If $H=\Sor(F)$, then
\begin{equation}
  \area{F}^{2}\geq\area{K}\area{H}.
  \label{eq:product-bound}
\end{equation}
\end{lemma}

\begin{proof}
Let $q_i$ be the boundary vertices, $e_i=q_{i+1}-q_i$, and
$n_i=N(e_i)$.  The half-plane hypothesis says
\[
  \inner{y}{n_i}\leq\inner{q_i}{n_i}
  =\det(q_i,q_{i+1})
  \qquad(y\in K).
\]
Taking the maximum over $K$ and summing gives, by shoelace,
\begin{equation}
  \Ssup(F)\leq2\area{F}.
  \label{eq:support-bound}
\end{equation}
Sorting changes only the order and placement of the directed edge vectors, so
$\Ssup(H)=\Ssup(F)$.  Since $\area{F}>0$, the directed edges of $F$ are not
all supported on a line, so $H=\Sor(F)$ has positive area and hence nonempty
interior; thus \Cref{lem:support-mixed} applies, and
\eqref{eq:support-mixed} gives
\[
  2V(H,K)=\Ssup(H)=\Ssup(F)\leq2\area{F}.
\]
Thus $0\leq V(H,K)\leq\area{F}$.  Combining this with
\Cref{lem:minkowski} yields
\[
  \area{F}^{2}\geq V(H,K)^2\geq\area{H}\area{K}.
\]
No strict support contact and no division are used.
\end{proof}

\begin{proposition}[Polygonal cap-union inequality]
\label{prop:polygonal-S}
Under the hypotheses of \Cref{lem:cap-interface}, the cap union satisfies
\eqref{eq:Szero}.
\end{proposition}

\begin{proof}
Write $k=\area{K}$, $u=\area{U}$, $c=\area{C}$, and
$H=\Sor(U)$.  By \Cref{lem:cap-interface}, $U$ is a legal input to
\Cref{lem:companion,lem:support-bridge}.  Since $\conv U=C$, they give
\[
  \area{H}+u\geq2c,
  \qquad
  u^2\geq k\area{H}.
\]
Therefore
\[
  u^2\geq k\area{H}\geq k(2c-u),
\]
which rearranges to $u^2+ku\geq2kc$.  This is
\eqref{eq:Szero}.
\end{proof}

\section{Compact export and the polygon theorem}
\label{sec:compact-export}

The remaining issue in \Cref{thm:S-intro} is that $K$ need not be polygonal.
We use inner approximants with one fixed inball.  This gives the exact area
control needed for the finite nonconvex union of caps.

\begin{lemma}[Common-inball area estimate]
\label{lem:inball-area}
Let $A_n\subseteq A$ be compact convex sets, and suppose
$o+rB\subseteq A_n$ for a fixed $r>0$, where $B$ is the closed unit disk.  If
$d_H(A_n,A)\leq\delta$, then
\begin{equation}
  \area{A_n}\leq\area{A}
  \leq(1+\delta/r)^2\area{A_n}.
  \label{eq:area-estimate}
\end{equation}
\end{lemma}

\begin{proof}
Hausdorff containment and the inball give
\[
  A\subseteq A_n+\delta B
  \subseteq o+(1+\delta/r)(A_n-o).
\]
Indeed, $rB\subseteq A_n-o$, so
$\delta B\subseteq(\delta/r)(A_n-o)$, and convexity together with
$0\in A_n-o$ identifies the resulting Minkowski sum with the displayed
dilate.  Taking areas proves \eqref{eq:area-estimate}.
\end{proof}

\begin{proof}[Proof of \Cref{thm:S-intro}]
Write $k=\area{K}$, $u=\area{U}$, and $c=\area{C}$.  First suppose $k>0$.
Then $K$ has nonempty interior.  Choose a nondegenerate closed triangle
$T\subset\operatorname{int}K$ containing a ball $o+rB$, $r>0$.  For each
$n\geq1$, choose a finite $1/n$-net $D_n\subset K$ and put
\begin{equation}
  K_n=\conv\bigl(T\cup D_1\cup\cdots\cup D_n\bigr).
  \label{eq:Kn}
\end{equation}
Then $K_n$ is a positive-area convex polygon,
\[
  K_n\subseteq K_{n+1}\subseteq K,
  \qquad d_H(K_n,K)\leq1/n,
\]
and every $K_n$ contains $o+rB$.  By \Cref{lem:inball-area},
$k_n=\area{K_n}\to k$.

Keep the generator set $X$ fixed and define
\[
  P_{n,x}=\conv(K_n\cup\{x\}),
  \qquad
  P_x=\conv(K\cup\{x\}),
  \qquad
  U_n=\bigcup_{x\in X}P_{n,x}.
\]
Replacing the $K$-entry in a convex combination by a nearest point of $K_n$
shows
\[
  d_H(P_{n,x},P_x)\leq d_H(K_n,K).
\]
Both bodies contain the fixed inball, so \Cref{lem:inball-area} gives
$\area{P_{n,x}}\to\area{P_x}$ for every $x$.  Moreover
$P_{n,x}\subseteq P_x$, and finiteness of $X$ yields
\begin{equation}
  0\leq\area{U}-\area{U_n}
  \leq\sum_{x\in X}
    \bigl(\area{P_x}-\area{P_{n,x}}\bigr)
  \longrightarrow0.
  \label{eq:union-convergence}
\end{equation}
Thus $u_n=\area{U_n}\to u$.  The generator set has not changed, and
$X\subseteq U_n\subseteq C$, so $\conv U_n=C$ for every $n$.

Apply \Cref{prop:polygonal-S} to $(K_n,X)$:
\[
  u_n^2+k_n u_n\geq 2k_n c.
\]
Letting $n\to\infty$ proves \eqref{eq:Szero} when $k>0$.  If $k=0$, including
the cases where $K$ is a point or segment, \eqref{eq:Szero} is simply
$u^2\geq0$.  These branches exhaust the theorem's domain.
\end{proof}

We now pass from the convex-body theorem to polygons.

\begin{proof}[Proof of \Cref{thm:C1-intro}]
Put
\[
  f=\area{F},\qquad k=\area{K(F)},\qquad c=\area{C(F)}.
\]
If the kernel is empty or $k=0$, \eqref{eq:C1} reduces to $f^2\geq0$.
Assume $k>0$.  The visibility kernel is then a nonempty compact convex set:
it is a closed subset of the compact region $F$, and convexity follows from
the associativity of segments applied to two star centers.  Let $X$ be the
finite set of hull vertices of $F$.  For each $x\in X$ and $z\in K(F)$,
visibility gives $[z,x]\subseteq F$.  Consequently
\[
  \conv(K(F)\cup\{x\})\subseteq F,
  \qquad
  U\subseteq F.
\]
Also $X\subseteq U$, so $\conv U=C(F)$.  With $u=\area{U}$,
\Cref{thm:S-intro} gives $u^2+ku-2kc\geq0$, while $f\geq u$.  The exact
identity
\begin{equation}
  f^2+kf-2kc
  =\bigl(u^2+ku-2kc\bigr)+(f-u)(f+u+k)
  \label{eq:monotone-transfer}
\end{equation}
therefore shows $f^2+kf\geq2kc$.  Rearranging is precisely
\eqref{eq:C1}.  Dividing the subtraction-free form by $fc>0$ gives
$A+AG\geq2G$, hence \eqref{eq:ratio-form}.
\end{proof}

\section{Sharpness and comparison with Nakano}
\label{sec:sharpness}

The cap-union theorem is attained trivially when $K=C$ and $X$ generates
$C$.  More importantly, the coefficient two in the polygon theorem is
attained by nonconvex polygons with positive-area kernels.

\begin{proposition}[A nonconvex equality family]
\label{prop:sibley-family}
For $0<a<1$, let $F_a$ be the simple polygonal region with cyclic vertices
\[
  (-1,0),\quad (0,a),\quad (1,0),\quad (0,1),
\]
listed counterclockwise.
Then
\[
  \area{C(F_a)}=1,
  \qquad
  \area{F_a}=1-a,
  \qquad
  \area{K(F_a)}=\frac{(1-a)^2}{1+a},
\]
and equality holds in \eqref{eq:C1}.
\end{proposition}

\begin{proof}
The first two area formulas follow directly by shoelace: $F_a$ is the hull
triangle with the triangular indentation of area $a$ removed.  The four
vertices will be denoted by $L,I,R,T$ in the displayed order.  The diagonal
$IT$ divides $F_a$ into the two convex triangles $LIT$ and $IRT$.  If
$z=(x,y)$ lies in the left triangle, then it already sees that entire
triangle; it sees the right triangle exactly when the segment $zR$ crosses
$IT$ at or above $I$.  This is the inequality $y\geq a(1-x)$.  The reflected
argument on the right gives $y\geq a(1+x)$.  The two upper edges require
$y\leq1+x$ and $y\leq1-x$.  Thus, without invoking a general kernel
half-plane theorem, the visibility conditions reduce exactly to
\[
  a(1+\lvert x\rvert)\leq y\leq1-\lvert x\rvert.
\]
Consequently $\lvert x\rvert\leq(1-a)/(1+a)$, and the kernel is a kite with
vertical diagonal $1-a$ and horizontal diagonal
$2(1-a)/(1+a)$.  Its area is therefore $(1-a)^2/(1+a)$, as displayed.  It
follows that
\[
  G(F_a)=\frac{1-a}{1+a},
  \qquad A(F_a)=1-a,
\]
and therefore
\[
  G(F_a)=\frac{A(F_a)}{2-A(F_a)}.
\]
Equivalently, both sides of \eqref{eq:C1} equal
$2a(1-a)^2/(1+a)$.  Since $k(c-f)>0$, replacing $2$ in
\eqref{eq:C1} by any larger universal coefficient fails on every member of
this family.
\end{proof}

Up to cyclic ordering, the displayed family $F_a$ is exactly Sibley's
Figure~21~\cite[Fig.~21]{Sibley}; Sibley already computes
$G(F_a)=(1-a)/(1+a)$ and $E(F_a)=1-a$.  His Figures~1(a) and~5 show a
different shallowly indented polygon, while Figure~14 is a many-vertex
parabolic family.  The proof above is retained to make the present argument
self-contained.  For the present theorem, Sibley's Figure~21 formulas
immediately give $G=E/(2-E)$, so this family produces equality in C1 and proves
sharpness of its coefficient two; no novelty is claimed for the family, the
formulas, or this algebraic substitution.

Nakano's theorem says
\[
  G\leq A,
  \qquad\text{equivalently}\qquad
  \area{F}\bigl(\area{F}-\area{K}\bigr)
  \geq \area{K}\bigl(\area{C}-\area{F}\bigr).
\]
Thus \eqref{eq:C1} doubles the coefficient on the hull-deficit term.  In the
odds coordinate $\tau(t)=(1-t)/t$, and when $G>0$ Nakano's statement is
$\tau(G)\geq\tau(A)$, whereas \eqref{eq:C1} is
\[
  \tau(G)\geq2\tau(A).
\]
When $G=0$, both inequalities are automatic and this odds notation is not
used.
Nakano's proof proceeds through a kernel-adapted anisotropic perimeter and a
Wulff-type inequality~\cite{Nakano}.  The argument here gives an alternative
route: it constructs the cap union $U$, rearranges its boundary into a convex
companion, and invokes only the planar mixed-area inequality.  We make no
priority claim about the underlying classical tools.

For a simple polygonal region $F$ with positive-area visibility kernel
$K=K(F)$, and with $C=\conv F$, the exact provenance of the short synthesis
can be displayed in one line.  Nakano's support estimate, invariance under edge
sorting, and the planar Wulff--Minkowski inequality give
\[
  \Per_K(F)\leq2\area{F},\qquad
  \Per_K(\Sor F)=\Per_K(F),\qquad
  \area{F}^{2}\geq\area{K}\area{\Sor F},
\]
while Siegel's singleton-cycle Theorem~2 gives
$\area{\Sor F}+\area{F}\geq2\area{C}$
\cite{Nakano,Siegel}.  Multiplying this inequality by $\area{K}$ and using
$\area{K}\area{\Sor F}\leq\area{F}^{2}$ gives the exact chain
\[
  2\area{K}\area{C}
  \leq\area{K}\bigl(\area{\Sor F}+\area{F}\bigr)
  \leq\area{F}^{2}+\area{K}\area{F},
\]
which is \eqref{eq:C1}.  The BBMP/Pach inequality alone gives only
$\area{C}\leq\area{\Sor F}$ and hence the factor-one bound.  Thus novelty is
claimed for the sharp visibility-kernel consequence, its general cap-union
and degenerate formulation, and the formalization---not for any component
inequality.

A referee may reasonably ask whether \eqref{eq:C1} already follows from
Nakano's own sandwich.  For a simple polygon whose kernel has positive area he
proves~\cite[Proposition~2]{Nakano} that
\[
  P_K(F)^{2}\leq G(F)\leq P_K(F),
  \qquad
  P_K(F)=\Per_K(K)/\Per_K(F),
\]
and this does not give \eqref{eq:ratio-form}.  Nakano's proof also gives
$P_K\leq\sqrt{GA}$.  Combining this with $G\leq P_K$ recovers only
$G\leq A$ (and, together with $P_K^2\leq G$, also $P_K\leq A$), not the
sharper bound $G\leq A/(2-A)$.  What the present proof
supplies, and what no rearrangement of the sandwich supplies, is the
\emph{hull} term: the companion inequality \eqref{eq:CVX} is where $\area{C}$
enters, and it is the source of the factor two.  Two further differences are
worth recording.  Nakano's sandwich presupposes $\area{K(F)}>0$, since
$P_K$ is undefined otherwise, whereas \eqref{eq:C1} includes empty and
zero-area kernels.  And he records no nonconvex equality case, writing that he
does ``not attempt to characterize all equality cases'', whereas
\Cref{prop:sibley-family} exhibits a one-parameter family of them.

It is important not to confuse two appearances of the number $2$.
The coefficient in \eqref{eq:C1} is the sharp \emph{area-deficit}
coefficient proved here.  Nakano's separate bound $G\leq2P_{\mathrm{per}}$
concerns the perimeter ratio
$P_{\mathrm{per}}(F)=\Per(C(F))/\Per(F)$ and leads to a different open
constant $\alpha^*$, discussed only in \Cref{sec:roadmap}.

\section{Formal verification and reproducibility}
\label{sec:formal}

The accompanying formalization uses Lean~4.33.0 and Mathlib~v4.33.0
\cite{LeanRelease,MathlibRelease}; for general background on the systems, see
\cite{Lean4,Mathlib}.  The reader-facing entry module is
\begin{center}
  \texttt{MathLab.P5.KernelDeficit}
\end{center}
at
\begin{center}
  \texttt{lean/MathLab/MathLab/P5/KernelDeficit.lean}.
\end{center}
The short reproduction guide is
\begin{center}
  \texttt{lean/MathLab/P5-KERNEL-DEFICIT.md}.
\end{center}
The maintained public repository is
\begin{center}
  \url{https://github.com/SilverAsh7/p5-kernel-deficit-lean}.
\end{center}
An earlier audited production-only public snapshot is commit
\begin{center}
  \texttt{64e81a503ee4d78875cb9cededdda18995f30007}.
\end{center}
A source archive of that historical snapshot is available at
\begin{center}
  \url{https://github.com/SilverAsh7/p5-kernel-deficit-lean/archive/64e81a503ee4d78875cb9cededdda18995f30007.zip}.
\end{center}
The proof-complete public \texttt{main} snapshot is commit
\begin{center}
  \texttt{e62a622ed1ff1a70f4a10a860b4088ee35ee6b68}.
\end{center}
Its source archive is available at
\begin{center}
  \url{https://github.com/SilverAsh7/p5-kernel-deficit-lean/archive/e62a622ed1ff1a70f4a10a860b4088ee35ee6b68.zip}.
\end{center}

The immutable source artifact of record for this deposit is
\begin{center}
  \href{https://doi.org/10.5281/zenodo.22062707}
  {\texttt{P5-kernel-deficit-lean-v6.zip}}.
\end{center}
Its SHA-256 is the concatenation of the following two lines:
\begin{center}
\texttt{6DD940DD393FDC2D9336C3D30C8E9C7B}\\[-0.2ex]
\texttt{F34A9DF3DC4E598731B869AD4A12E10D}.
\end{center}
The 1,587,378-byte archive contains 219 payload files, including 204 Lean
sources: the pinned toolchain and complete transitive source closure of the
production, research, and smoke-test aggregates, every imported scratch
module, and the source-provenanced vendored Schoenflies--Jordan dependency;
build caches are omitted.

The deposited archive and commit \texttt{e62a622} contain identical executable
Lean code but are intentionally not byte-for-byte identical.  The repository
additionally carries \path{.gitattributes} and
\path{.github/workflows/lean.yml}; its \path{README.md} and
\path{CITATION.cff} carry repository-specific presentation and citation
metadata; and five Lean source files differ only in comments.  Thus the
SHA-256 above identifies the exact deposited artifact, while the commit hash
identifies the exact public repository snapshot.

The separate numerical material cited in the final section is packaged as the
archive
\begin{center}
\href{https://doi.org/10.5281/zenodo.22062707}
{\texttt{kernel-deficit-v6-supplement.zip}}.
\end{center}
It is distributed with this preprint through the stable Zenodo concept record.
Its SHA-256 is the concatenation of the following two lines:
\begin{center}
\texttt{8FA9C2CE0A35FF67DB399E1D24F33EFD}\\[-0.2ex]
\texttt{40286369C83897824A8ABE3FD5E3BBB4}.
\end{center}
It contains all three polygon
vertex lists, the exact rational certificate data, and a minimal exact
verifier; it is independent of the Lean build.

\newpage
The public headline declarations are:
\begin{itemize}[leftmargin=2em]
  \item \texttt{MathLab.P5.S\_zero}, the real-valued cap-union theorem
        \eqref{eq:Szero};
  \item \texttt{MathLab.P5.c1}, the subtraction-free statement
        $2kc\leq f^2+kf$ for the literal enclosed region, its geometric
        visibility kernel, and its vertex hull; and
  \item
        \texttt{MathLab.P5.kernelDeficit\_dominates\_twice\_hullDeficit},
        the paper-facing deficit form \eqref{eq:C1}.
\end{itemize}
No nonempty-kernel or positive-kernel-area hypothesis appears in the polygon
theorems.

From \texttt{lean/MathLab/}, the reader-facing checks are
\begin{verbatim}
lake build MathLab.P5.KernelDeficit
lake env lean MathLab/P5/KernelDeficit.lean
lake build MathLab.Research
lake build
\end{verbatim}
At the completed v6 source snapshot, direct compilation passed, the 8,715-job
package-target build passed, the 8,855-job research-aggregate build passed,
and the full 8,721-job build passed.  The entry module prints the axioms of the
two polygon
declarations, and the exact output for each is
\begin{verbatim}
[propext, Classical.choice, Quot.sound]
\end{verbatim}
The same exact axiom set is reported for \texttt{MathLab.P5.S\_zero}.  There
is no \texttt{sorryAx} and no project-specific axiom in any of these public
theorem dependency closures.

The statement audit is part of the claim, not merely a software check.  In
particular, the formal polygon is a positive simple polygonal cycle; its
region is the literal bounded enclosed region; $K$ is its geometric visibility
kernel; $C$ is the convex hull of the listed vertices; all quantities are
ordinary planar Lebesgue areas; and empty or lower-dimensional kernels are
retained.  The subtraction-free statement is ring-equivalent to
\eqref{eq:C1}, without a sign or division side condition.

Three points on which such an audit is ordinarily taken on trust are here
discharged by machine, in \texttt{MathLab/P5/KernelDefinitionBridge.lean}.
The formal kernel $\{x:[x,y]\subseteq U\text{ for every }y\in U\}$ omits the
conjunct $z\in F$ carried by the definition of $K(F)$ in
\Cref{sec:introduction}; \texttt{c1\_paperKernel} proves the subtraction-free
form of \eqref{eq:C1} for the literal definition with that conjunct restored,
and the ring equivalence above gives the displayed deficit form, so nothing
turns on the omission.  \texttt{hull\_identity} discharges the identification, asserted in
\Cref{sec:preliminaries}, of the convex hull of a simple polygonal region with
the convex hull of its listed vertices.  And \texttt{nonvacuity} exhibits a
cycle whose region, hull, and kernel all have area $1/2$, at which
\eqref{eq:C1} holds with equality; since \eqref{eq:C1} is homogeneous of
degree two, this also pins the measure as ordinary planar area rather than a
rescaling of it.  Each reports the same three axioms.  In particular the
visibility kernel, and not the intersection of the inner half-planes of the
edge lines, is the object throughout: the two agree for a simple polygon, but
that agreement is a classical theorem which is not invoked in
Sections~1--8 and is not formalized in the Lean development.  The distinct
numerical certificate argument in the final section does invoke the theorem,
with its source stated there.

The paper proof and the formal proof have the same mathematical spine but do
not expose identical intermediate generalities, and it is worth stating
precisely which lemmas of \Cref{sec:rearrangement} are machine-checked.

\Cref{lem:sort-max}, which carries the factor of two, \emph{is}.  It is the
pair \texttt{edgeSortMax} and \path{exists_isCyclicCCWEdgeSort}, both
sorry-free with the same three axioms and both post-proof statement-audited.
The formal proof is independent of the one given above, proceeding by a signed
determinant sorting induction and a sorted-trace area identity rather than by
the Eulerian split and the limit.  The formal statement quantifies over
\emph{every} legal cyclic sorting permutation, retains zero edges in the
permuted edge multiset but leaves their positions unrestricted because the
sorting predicate ignores them, leaves codirected edges unconsolidated, and
reads $\area{\Sor(Q)}$ as the area
of the convex hull of the companion's partial-sum vertices; each of these
readings agrees with the text's.

\Cref{lem:companion}, the sorted-companion inequality \eqref{eq:CVX} in its
general form, is now machine-checked as
\texttt{PolygonRearrangement.cvx} in the explicit research extension
\path{MathLab/Geometry/PolygonRearrangementResearch.lean}.  Its general
positive-ear removal, positive-triangulation, area-identification, and virtual
hull inputs---blueprint B95, B94, and B91, leading to B87---are all
\emph{FORMALIZED+}: the kernel checks their proofs, their final statements have
passed post-proof audits, and each public declaration reports exactly
\texttt{propext}, \texttt{Classical.choice}, and \texttt{Quot.sound}.  Both
the production aggregate and the full \texttt{MathLab.Research} aggregate are
therefore admission-free.

Two scope boundaries are deliberate.  The B91 declaration
\path{exists_virtualHullBlockReversal} proves the existential certificate
actually consumed downstream: there is an edge permutation with the required
signed-area identity.  Its type does not claim that the witness permutation is
literally assembled from reversed hull-chain blocks, nor that the permuted
walk is simple.  Likewise, B94's \path{PositiveTriangulation} is an exact
positive area certificate: listed-vertex triangles cover the enclosed region,
overlap only on null sets, and telescope to shoelace area.  It is not a full
simplicial-complex triangulation with a face-count or incidence theorem.

The general formal proof imports the source-provenanced Schoenflies--Jordan
development vendored in the v6 archive; those sources are kernel-checked and
introduce no additional axiom.  This route is distinct from the repaired
informal proof above.  The manuscript's B114/B116 route derives the hull order
directly from polygonal-Jordan and crosscut facts in
\Cref{lem:hull-polygon,lem:crosscut,lem:hull-order}; the former unsupported
interleaved-arc/Jordan--Schoenflies step has been removed entirely.

Finally, the headline B12/C1 closure did not acquire a new dependency through
this repair.  It already used the audited strict-radial B99 certificate,
together with the formalized signed edge-sort maximum, polygonal
support/mixed-area bridge,
planar Brunn--Minkowski theorem, and compact export.  General B91, B94, and B95
were outside that recursive closure; their former gaps therefore limited only
the broader general rearrangement claim, not either main theorem.

\section{Discussion and exact status boundary}
\label{sec:discussion}

The main inequality is closed, but its structure leaves several natural
questions.

\begin{enumerate}[leftmargin=2em]
  \item \emph{Equality.}  For $\area{K(F)}>0$, identity
        \eqref{eq:monotone-transfer} shows that equality in C1 forces both
        $\area{F}=\area{U}$ and equality in $S^\circ$.  A complete
        reader-facing classification would require a substantially longer
        argument, which is neither reproduced nor claimed here.  Natural
        next targets are an explicit moduli or block-local normal form and
        equality and stability for the general cap-union theorem $S^\circ$.
  \item \emph{Stability.}  The proof has two visible deficits: the
        companion-area deficit and the mixed-area deficit.  Quantifying their
        interaction may give a stability version of C1, but no such theorem is
        claimed here.
  \item \emph{Dimension and infinite generators.}  Cyclic edge sorting is
        intrinsically planar and finiteness is used both in the polygonal
        boundary arrangement and in \eqref{eq:union-convergence}.  Neither a
        higher-dimensional analogue nor an arbitrary compact generator set is
        asserted.
\end{enumerate}

For clarity, the status boundary is as follows.  Theorems
\ref{thm:S-intro} and \ref{thm:C1-intro}, their degeneracy branches, and the
specialized Lean route just described are proved, kernel-checked, and
post-proof statement-audited; they additionally passed the automated referee
and citation-verification passes described in the declaration at the end of
the paper.  No complete equality-classification theorem beyond the displayed
consequences is claimed in this manuscript.  The explicit-moduli and
general-$S^\circ$ equality/stability questions remain outside its proved
scope.  Finite computations mentioned in
the next section do not enter the proof of either theorem.

\section{Roadmap to the exact optimum in P5}
\label{sec:roadmap}

\noindent\textbf{Separation from the proved results.}
Nothing in this section is used in the proofs of
\Cref{thm:C1-intro,thm:S-intro}.

This final section concerns Nakano's distinct perimeter problem.  Define
\[
  P_{\mathrm{per}}(F)=\frac{\Per(C(F))}{\Per(F)},
  \qquad
  \alpha^*=\sup_F\frac{G(F)}{P_{\mathrm{per}}(F)}.
\]
Nakano proved $1612/1575\leq\alpha^*\leq2$ and asked for the exact value
\cite{Nakano}.  The present state, with each status stated explicitly, is
summarized in \Cref{tab:alpha-status}.

The three exact certificates below compute the intersection of the closed
inward half-planes of the directed boundary edges.  For a counterclockwise
simple polygon, Nakano records that this intersection is exactly the
every-point visibility kernel~\cite[p.~2]{Nakano}.  Thus the computed kernel
area is exactly the numerator of $G$, not an auxiliary half-plane estimate.
This classical bridge is load-bearing only in this numerical section; the
proofs in Sections~1--8 remain in the visibility-kernel register throughout.

\begin{table}[ht]
\centering
\small
\begin{tabular}{@{}p{0.19\textwidth}p{0.73\textwidth}@{}}
\toprule
\textbf{Status} & \textbf{Statement or task} \\
\midrule
Proved; certificate independently recomputed
& An explicit $503$-gon with $250$ notches and integer coordinates bounded
  by $3\cdot10^9$ gives the rigorous exact-rational certificate
  \[
    \alpha^*>1.06107423573065456901,
  \]
  whose exact fraction has a $3390$-digit numerator.  Together with Nakano's
  theorem, the certified interval is
  $1.06107423573065456901<\alpha^*\leq2$.
  The image of that polygon under $\operatorname{diag}(10^{6},1)$ is again an
  explicit integer $503$-gon whose \emph{own} ratio $G/P_{\mathrm{per}}$ lies
  in a certified rational interval about $1.06107349321471418331$.  That
  value exceeds the earlier $403$-gon certificate
  $\alpha^*>1.06099297446846380587$, which is therefore exceeded by a
  single explicit polygon and not only through a compression limit, while
  falling short of the displayed lower end by $7.43\cdot10^{-7}$.  The
  earlier certificate remains valid and is superseded only in value.  All
  three vertex lists, exact fractions, and the exact verifier are in the
  accompanying \path{kernel-deficit-v6-supplement.zip}, available through the
  \href{https://doi.org/10.5281/zenodo.22062707}{Zenodo concept record}. \\
\addlinespace
Proved; compression proof below, finite lemma in companion note
& For every polygon $F$ and nonzero direction $u$, directional anisotropic
  compression gives
  $\alpha^*\geq \Phi_u(F)
   \defeq G(F)N_F(u)/N_C(u)$, where
  $N_R(u)=\sum_{e\in\mathcal E(R)}\lvert\inner{e}{u}\rvert$.
  The maximum over $u$ is attained at a normal to an edge of $F$ or of
  $C(F)$ and is therefore obtained by comparing finitely many candidates.
  Between consecutive edge normals its derivative has constant sign, namely
  the sign of $-\det(a,b)$ when
  $N_F/N_C=\inner{a}{u}/\inner{b}{u}$; for rational vertex data each candidate
  ratio $N_F(u)/N_C(u)$ is computed exactly in rational arithmetic.  This
  standard finiteness argument is not claimed as new.  The full proof is
  \path{paper/alpha-star-lower-bound.md}, \S3, Proposition~3; see also
  \Cref{prop:compression-lower}. \\
\addlinespace
Verified finite computation
& Exact rational certificates were produced for the optimized displayed combs
  at $k=1,\ldots,120$ and $k=160,200,250$; the values at $k=121,\ldots,159$
  and $k=161,\ldots,199$ were not computed.  Each certificate is a lower
  bound for the optimum of its own $k$-notch family rather than that optimum,
  since the underlying parameters come from a numerical search; the $123$
  certified values nevertheless increase strictly in $k$ under exact rational
  comparison.  Unrestricted polygon searches covered only $4\leq n\leq10$,
  and bounded two-sided and other adversarial families were also tested.
  Each explicit polygon supplies a valid lower bound; the absence of a better
  polygon in a finite search is evidence only. \\
\addlinespace
Conjectured
& The optimized comb values appear monotone and a continuum variational model
  suggests the comb-family limit
  $1.0614097453924$.  Neither monotonicity, global comb optimality, continuum
  convergence, nor that limiting value is proved.  Even a proof of this value
  would determine only the comb family, not $\alpha^*$. \\
\addlinespace
Required, currently unproved
& A genuinely global upper inequality or a monotone normal-form theorem must
  control arbitrary polygons.  The verified reductions below do not provide
  that missing estimate. \\
\bottomrule
\end{tabular}
\caption{Status separation for the exact-$\alpha^*$ program.}
\label{tab:alpha-status}
\end{table}

The FORMALIZED+ area theorem C1 does not determine $\alpha^*$: C1 compares
areas, while $P_{\mathrm{per}}$ contains boundary length.  No unproved
perimeter control is being smuggled into the formal result.

\subsection*{Verified reductions}

We record the short compression and slice arguments so their status is not
left to a roadmap label.

Nakano already used the vertical special case of this limiting compression on
his explicit pentagon to obtain $1612/1575$~\cite[\S7]{Nakano}.  The next
proposition gives the directionwise form; Proposition~R then identifies its
maximum with the full affine-orbit supremum.

\begin{proposition}[Directional compression]
\label{prop:compression-lower}
For every simple polygonal region $F$ and every $u\neq0$,
\[
  \alpha^*\geq\Phi_u(F)
  \defeq G(F)\frac{N_F(u)}{N_C(u)},
  \qquad
  N_R(u)=\sum_{e\in\mathcal E(R)}\lvert\inner{e}{u}\rvert.
\]
\end{proposition}

\begin{proof}
Normalize $u$ and let $S_t$ fix the $u$ direction while multiplying the
orthogonal direction by $t>0$.  Affine covariance of the visibility kernel
gives $G(S_tF)=G(F)$.  For every edge $e$,
$\lvert S_te\rvert\to\lvert\inner{e}{u}\rvert$ as $t\downarrow0$; hence
\[
 \frac{G(S_tF)}{P_{\mathrm{per}}(S_tF)}
 =G(F)\frac{\Per(S_tF)}{\Per(S_tC)}
 \longrightarrow \Phi_u(F).
\]
The denominator tends to $N_C(u)>0$, twice the width of $C$ in direction
$u$.  Every $S_tF$ is a genuine simple polygon, so taking the limit proves
the claim.
\end{proof}

\begin{proposition}[Compression completeness (R)]
\label{prop:compression-complete}
The directional reduction is exact:
\[
  \alpha^*=\sup_F\max_{u\neq0}\Phi_u(F).
\]
\end{proposition}

\begin{proof}
For any polygonal region $R$ with directed edge vectors $e$ and with
$d\sigma$ denoting arc-length measure on $S^1$, the elementary edgewise
identity
\[
 \Per(R)=\sum_e\lvert e\rvert
 =\frac14\int_{S^1}\sum_e\lvert\inner{e}{u}\rvert\,d\sigma(u)
 =\frac14\int_{S^1}N_R(u)\,d\sigma(u)
\]
follows from
$\int_{S^1}\lvert\inner{e}{u}\rvert\,d\sigma(u)=4\lvert e\rvert$.
This identity does not require $R$ to be convex.  For a fixed $F$, it and the
mediant inequality give, for every invertible linear map $T$,
\[
 G(TF)\frac{\Per(TF)}{\Per(TC)}
 \leq G(F)\max_{u\neq0}\frac{N_F(u)}{N_C(u)}.
\]
Indeed $G(TF)=G(F)$, while the displayed identity applies to both perimeters
and
$N_{TR}(v)=N_R(T^{\mathsf T}v)$.  Conversely the compressions in
\Cref{prop:compression-lower} make the left side tend to each $\Phi_u(F)$.
Thus $\max_u\Phi_u(F)$ is precisely the supremum of
$G(TF)/P_{\mathrm{per}}(TF)$ over the affine orbit of $F$.  Taking the
supremum over $F$, a class already closed under invertible affine maps, gives
the displayed identity.
\end{proof}

The reverse inequality proved in \Cref{prop:compression-complete} transports no
upper bound by itself: it identifies the optimized directional functional
with an affine-orbit supremum rather than reducing the class of polygons.

The first identity below is the polygonal specialization of Banach's
indicatrix formula (equivalently, the one-dimensional coarea formula) applied
to the $x$-coordinate of a boundary parametrization~\cite{Banach}.  We include
the edgewise proof for completeness.

\begin{proposition}[Slice identity (S)]
\label{prop:slice-identity}
Let $F$ be a simple polygonal region, let $e_1=(1,0)$, and let $\nu_F(s)$ be
the number of connected components of the vertical slice
$F\cap\{x=s\}$.  Then (the values at the finitely many exceptional vertex
abscissas do not affect the integral)
\[
  N_F(e_1)=2\int_{\R}\nu_F(s)\,ds.
\]
If $W>0$ is the horizontal projection width of $F$, then
\[
  N_C(e_1)=2W,
  \qquad
  \Phi_{e_1}(F)=\frac{G(F)}{W}\int_{\R}\nu_F(s)\,ds.
\]
In particular the last formula has no factor $W^{-1}$ when $W=1$.
\end{proposition}

\begin{proof}
Away from the finite set of vertex abscissas, a vertical line meets only
nonvertical edge interiors.  Its boundary crossings alternate between entry
and exit along the line, so their number is $2\nu_F(s)$.  Integrating the
crossing count edge by edge gives
$\sum_e\lvert\Delta x_e\rvert=N_F(e_1)$, proving the first identity.  The
hull polygon supplied by \Cref{lem:hull-polygon} has one interval as its
generic vertical slice, so applying the same identity to $C$ gives
$N_C(e_1)=2W$.  Division and multiplication by $G(F)$ yield the last display.
\end{proof}

The first identity is deliberately G-free.  Its hull normalization follows
above from \Cref{lem:hull-polygon}, and multiplying the normalized identity by
$G(F)$ gives the displayed $\Phi$-form.  No internal project-status label is
needed for these consequences because their proofs are included here.

Both reductions were also subjected to counterexample searches before review,
including unrestricted polygons with $n\geq11$, nested and multiscale
notches, two-sided combs, several notched hull edges, and nonquadrilateral
hulls.  Those computations corroborate but do not prove the propositions.

\subsection*{Two proof branches}

With R and S now proved, a global proof still needs one of two genuinely
global inputs.

\begin{enumerate}[leftmargin=2em]
  \item \textbf{GLOBAL:} prove a universal upper inequality for the slice
        functional, with a sharp constant and all kernel degeneracies
        included; or
  \item \textbf{COMBIFICATION:} prove that every polygon can be transformed,
        without lowering the relevant objective, into a controlled comb or a
        rigorously specified broader normal form.
\end{enumerate}

Neither input is proved here, and neither is presently known to the author.  A
transformation that merely looks plausible on one-notch examples is especially
unsafe because repeated local defects contribute additively to the projected
boundary count.

The parallel finite-comb program has the following staged obligations:
\begin{description}[leftmargin=8.4em,style=nextline]
  \item[COMB-PARAM] give an exact parameter domain for every admissible
        $k$-notch comb, including simplicity and kernel combinatorics;
  \item[COMB-FORMULA] derive the exact rational objective throughout that
        domain;
  \item[K1-SEMIALG] solve the one-notch semialgebraic optimization exactly;
  \item[PENTAGON-MAX] prove, rather than infer from search, that the one-notch
        value is maximal among all pentagons.
\end{description}
The continuum branch then requires:
\begin{description}[leftmargin=10.8em,style=nextline]
  \item[CONT-DOMAIN] define and compactify the admissible continuum profile
        space;
  \item[DISCRETE-CONTINUUM] prove matching discrete-to-continuum bounds; and
  \item[CONT-OPT] solve the limiting variational problem with a certified
        optimizer and value.
\end{description}
The numerical target for that branch is the explicitly conjectural
functional
\[
 \Lambda_{\rm comb}
 =\sup_{(R,\phi)\in\mathcal D_{\rm comb}}
       \frac{2K_\infty(R,\phi)}{1+R}
       \bigl(1+S(R,\phi)\bigr),
\]
where
\begin{align*}
 S(R,\phi)
   &=\int_0^1\frac{dc}{R-1-\phi(c)-\phi(1-c)},\\
 K_\infty(R,\phi)
   &=\int_0^1
      \left(1+(R-1)x
        -2\max_{c\in[0,1]}(c+\phi(c)x)\right)_+\,dx.
\end{align*}
Here $\mathcal D_{\rm comb}$ denotes the as-yet-unformalized class of
admissible pairs $(R,\phi)$, and all occurrences of $c$ range over $[0,1]$.
The definition of that profile class, its boundary conditions, existence of an
optimizer, discrete-to-continuum convergence, and a rigorous upper enclosure
are all parts of the open problem; the displayed formula is not itself a
certification of the numerical value.

These steps would settle the comb family only.  Identifying their answer with
$\alpha^*$ still requires GLOBAL, COMBIFICATION, or a proof that a broader
family is extremal.  Accordingly, the only present numerical claim about the
global optimum is the certified interval displayed in
\Cref{tab:alpha-status}; $1.0614097453924$ is not asserted to be
$\alpha^*$.

\subsection*{Other remaining P5 work}

The exact-perimeter program is the largest open question, but it is not the
only unfinished structural work.  A complete P5 program would also:
\begin{enumerate}[leftmargin=2em]
  \item give a self-contained C1 equality classification and refine it into
        an explicit moduli or block-local normal form, and classify equality
        and prove quantitative stability for $S^\circ$;
  \item close the known finite-pocket boundary chambers, including alternate
        active-line orders and nontriangular hulls, and either prove or replace
        the currently conjectural medial-homothet pattern;
  \item treat nontriangular pockets and larger pocket counts without assuming
        the canonical square or triangle normal forms;
  \item if useful, strengthen the now-formalized general rearrangement
        certificates to a literal structural hull-block reversal and a full
        simplicial-complex triangulation.  Those stronger statements are not
        dependencies of the completed C1/$S^\circ$ formalization.
\end{enumerate}

\section*{Declaration of generative AI assistance}

The author used Claude (Anthropic) and ChatGPT/Codex (OpenAI) as research and
software-assistance tools for exploratory mathematics, Lean formalization,
manuscript organization, citation checking, and language editing.  The
AI-assisted workflow included separate referee and citation-verification
passes; these were automated checks, not human peer review.  Formalized claims
were additionally checked by the Lean kernel.  The author reviewed and
approved the final manuscript and accepts full responsibility for its
contents.

\end{document}